\documentclass[12pt,twoside,a4paper]{article}
 
\ifx\pdfpageheight\undefined\PassOptionsToPackage{dvips}{graphicx}%
\else%
\PassOptionsToPackage{pdftex}{graphicx}
\PassOptionsToPackage{pdftex}{color}
\fi

\usepackage{amsfonts}
\usepackage{amssymb}
\usepackage[french]{babel}
\usepackage[utf8]{inputenc}
\usepackage[T1]{fontenc}
\usepackage{url}

\ifx\pdfpageheight\undefined
\else
\usepackage[bookmarksopen=false,pdftex=true,breaklinks=true,%
    backref=page,pagebackref=true,plainpages=false,%
    hyperindex=true,pdfstartview=FitH,colorlinks=true,%
    pdfpagelabels=true,colorlinks=true,linkcolor=blue,citecolor=red]%
	{hyperref}
\fi

\DeclareSymbolFont{lasy}{U}{lasy}{m}{n}
\SetSymbolFont{lasy}{bold}{U}{lasy}{b}{n}
\let\Box\undefined
\DeclareMathSymbol\Box{\mathord}{lasy}{"32}

\newcounter{bidon}

\newcommand{\rdb}{\refstepcounter{bidon}}

\newtheorem{theorem}{Théorème}   

\newtheorem{proposition}{Proposition}[section]
\newtheorem{lemma}[proposition]{Lemme}

\newtheorem{remark}[proposition]{Remarque}

\newtheorem{definition}[proposition]{Définition}

\newtheorem{notation}[proposition]{Notation}

\newcommand {\junk}[1]{}

\newenvironment{proof}[1]{
\trivlist \item[\hskip \labelsep{\it #1}]}{\hfill\mbox{$\Box$}
\endtrivlist}

\makeatletter
\def\maketitle{\par
\begingroup
\def\@makefnmark{\hbox
to 0pt{$^{\@thefnmark}$\hss}}
\if@twocolumn
\twocolumn[\@maketitle]
\else \newpage
\global\@topnum\z@ \@maketitle
\fi\thispagestyle{plain}\@thanks
\endgroup
\setcounter{footnote}{0}
\let\maketitle\relax
\let\@maketitle\relax
\gdef\@thanks{}\gdef\@author{}\gdef\@title{}\let\thanks\relax} 
\makeatother

\def\.@{\char'76}

\def\egal {\hbox{~{\bf =}~}}
 
\def \noi {\noindent}
\def \ss {\smallskip}
\def \sni {\ss\noi}
\def \ms {\medskip}
\def \mni {\ms\noi}
\def \bs {\bigskip}
\def \bni {\bs\noi}
\def \hs {\qquad}

\def \snic#1 {\sni\centerline{$#1$}\ss}
\def \snif#1#2#3 {\vspace{#1}\noindent\centerline{$#3$}\vspace{#2}}

\def \npb{\nopagebreak}

\def \E {{{\sf E}}\,}
\def \Faux {{{\sf Faux}}\,}
\def \ou {{\;{\sf ou}\;}}

\def \im {\vdash}
\def \imp {\im\;}
\def \impl {\;\imp}

\def \bu {{$\bullet$}}

\def \hbu {{\bu\hs}}
\def \ax#1#2#3#4{\vspace{.1cm} \begin{tabular}{p{#1}ll}
               \hbu $#2\impl#3$ & $(#4)$ \end{tabular}}
\def \axi#1#2#3{\ax{12cm}{#1}{#2}{#3}}
\def \axio#1#2#3{\axi{#1}{#2}{#3} \vspace{.1cm}}

\def \sd#1#2 {\widetilde#1_#2 }

\def \ZZ{\mathbb{Z}}
\def \NN{\mathbb{N}}
\def \RR{\mathbb{R}}
\def \gx{{\bf x}}
\def \gy{{\bf y}}

\def \gu{{\bf u}}
\def \gt{{\bf t}}

\def \gA{{\bf A}}
\def \gB{{\bf B}}
\def \gK{{\bf K}}
\def \gL{{\bf L}}

\def \cD {{\cal D}}
\def \cI {{\cal I}}

\def \cP {{\cal P}}
\def \cM {{\cal M}}
\def \cT {{\cal T}}
\def \cL {{\cal L}}

\def \rP {{\rm P}}

\def \U#1{{\rm Ndz}_{#1}}
\def \J#1{{\rm Ide}_{#1}}

\def \Rzero {R_{=0}}
\def \Ru#1{{R_{\U#1}}}
\def \Rj#1{{R_{\J#1}}}

\def \Zg {{\ZZ[G]}}
\def \Izero {{\cI}_{=0}}
\def \Ij#1{{\cI}_{\J#1}}
\def \Mu#1{{\cM}_{\U#1}}

\def \Kxn {\gK[x_1,\ldots,x_n]}
\def \Lxn {\gL[x_1,\ldots,x_n]}

\def \Kx {\gK[\gx]}
\def \Lx {\gL[\gx]}

\def \Ax {{\rm A}}
\def \Ac {{\rm Aco}}

\newcommand \gui[1] {``#1''}

\newcommand \tho {théorème }

\newcommand \thos {théorèmes }

\def \ddk {dimension de Krull }

\def \dfn {définition }

\def \dfns {définitions }

\def \homo {homomorphisme }

\def \ndz {non diviseur de zéro }

\def \tdy {théorie \dy }

\def \tdys {théories \dys }

\def \dy {dynamique }
\def \dys {dynamiques }
\def \dyz {dynamique}
\def \dysz {dynamiques}

\def \rdyz {règle \dyz}
\def \rdys {règles \dys }
\def \rdysz {règles \dysz}

\def \evd {évaluation \dy }

\def \evdz {évaluation \dyz}
\def \evdsz {évaluations \dysz}

\def \sad {structure algébrique \dy }
\def \sads {structures algébriques \dys }
\def \sadz {structure algébrique \dyz}

\def \tfp {théorie formelle du premier ordre }

\def \cacz {corps algébriquement clos}

\def \tcg {théorème de complétude de G\"odel } 
\def \tcgz {théorème de complétude de G\"odel}

\def \nst {Nullstellensatz }
\def \nstz {Nullstellensatz}
\def \nsts {Nullstellens\"atze }
\def \nstsz {Nullstellens\"atze}

\def \ssi {si et seulement si }
\def \cnes {condition nécessaire et suffisante }
\def \cad {c'est-à-dire }
\def \spdg {sans perte de généralité }

\def \clama {\maths classiques }
\def \clamaz {\maths classiques}
\def \maths {mathématiques }
\def \cofz {constructif}

\def \ideps {idéaux premiers }
\def \idepsz {idéaux premiers}

\def \pol {polynôme }
\def \pols {polynômes }

\begin{document}
\title{ Dimension de Krull, Nullstellens\"atze\\ 
et \'Evaluation Dynamique }
\author{
Henri Lombardi
\thanks{
Equipe de Mathématiques, CNRS UMR 6623, UFR des Sciences et
Techniques,
Université de Franche-Comté, 25 030 BESANCON cedex, FRANCE,
email: henri.lombardi@univ-fcomte.fr}
}
\date{septembre 2000}
\maketitle

Cet article est paru dans la revue {\em Mathematische Zeitschrift},
{\bf 242}, 23--46, (2002). La version finale a été donnée en septembre 2000.

J'ai ajouté une petite {\bf Note historique} en postface page \pageref{notehistorique}, et une {\bf Table des matières}.

\begin{abstract}
Nous démontrons un \nst qui établit une équivalence entre 
l'existence d'une identité algébrique d'un certain type, d'une 
part, et l'impossibilité de trouver une suite croissante de 
variétés irréductibles répondant à certaines contraintes 
d'autre part. De ce point de vue le \nst usuel correspond au cas des 
variétés réduites à un point. 
Nous établissons aussi un \nst formel du même type,  en relation 
avec les suites croissantes d'\idepsz. Un cas particulier important est 
donné par la notion de suite pseudo régulière, plus générale 
que la notion de suite régulière. Nous obtenons de cette manière 
une nouvelle caractérisation de la dimension de Krull d'un anneau~: 
un anneau~a une dimension de Krull $\geq \ell$ \ssi il existe une suite 
pseudo régulière de longueur $\ell$ dans l'anneau.
Dans les cas où ces résultats peuvent avoir une signification 
constructive précise, nos démonstrations y aboutissent constructivement.  
Nous pensons avoir donné ainsi quelques éléments en vue d'une 
interprétation constructive de la théorie de la dimension de Krull 
des anneaux commutatifs. Notre méthode utilise la notion de \sad 
introduite dans des articles précédents.

\bni \centerline {Abstract~~~~~~~~~}

We prove constructively a \nst giving an equivalence between the 
existence of a certain kind of algebraic identity on one hand, 
and the impossibility of finding an increasing sequence of irreducible 
varieties obeying certain constraints on the other hand.  
The ususal \nst corresponds to the case of varieties that are reduced 
to a point.
We settle also a similar formal \nst related to increasing sequences of 
primes. 
An important particular case is given by the notion of
pseudo regular sequence. 
This allows a new characterisation of the Krull dimension of a ring. 
This characterisation via pseudo regular sequences is elementary and 
constructive.
Our method uses dynamical algebraic structures which were introduced in 
previous papers. 
\end{abstract}
\mni MSC : 13C15, 03F65, 13A15, 13E05

\mni Mots clés :  Dimension de Krull, Nullstellensatz, 
Mathématiques constructives, Structures algébriques \dysz.

\mni Key Words : Krull Dimension, Nullstellensatz, Constructive 
Mathematics, Dynamical Algebraic Structures.



\section*{Introduction} \label{sec Introduction}
\addcontentsline{toc}{section}{Introduction}

\markboth{Introduction}{Introduction}
Dans cet article nous nous intéressons à une approche constructive 
des anneaux commutatifs dont la dimension de Krull est au moins égale 
à $\ell$. Le concept est difficile à manipuler constructivement car 
en général on n'a pas un accès explicite aux \ideps d'un anneau 
commutatif. 

\paragraph{Suites pseudo régulières et dimension de Krull}~

\noindent 
Nous introduisons la notion de suite pseudo régulière, qui donne un 
test explicite concernant la dimension de Krull d'un anneau commutatif. 
Nous appelons {\em suite pseudo régulière de $\ell$ éléments} 
(dans un anneau commutatif $\gA$) une suite $(x_1,\ldots,x_\ell)$ telle 
que pour tous $a_1,\ldots,a_\ell\in \gA$ et tous $m_1,\ldots,m_\ell\in 
\NN$ on ait
$$ x_1^{m_1}(x_2^{m_2}\cdots(x_\ell^{m_\ell} (1+a_\ell x_\ell) + 
\cdots+a_2x_2) + a_1x_1) \neq  0
$$
Une suite pseudo régulière de longueur $\ell$ permet de spécifier 
incomplètement une chaîne strictement croissante de $\ell+1$ 
\idepsz, conformément au théorème suivant que nous démontrons 
dans la section  \ref{sec regseq}.

\medskip
\noindent
{\bf Théorème \ref{th.pseudoreg}}
{\it {\em (suites pseudo régulières et chaînes croissantes 
d'\idepsz).}
Dans un anneau $\gA$ une suite $(x_1,\ldots,x_\ell)$ est pseudo 
régulière \ssi il existe  $\ell+1$ \ideps $P_1\subset \cdots\subset 
P_\ell\subset P_{\ell+1}$ avec $x_1\in P_2\setminus P_1$, $x_2\in 
P_3\setminus P_2, \ldots$  $x_\ell\in P_{\ell+1}\setminus P_\ell$.
 }

\medskip Etant donnée une suite pseudo régulière de longueur 
$\ell$ notre démonstration ne peut pas fournir un procédé de construction 
explicite des \ideps $P_1,\ldots,P_{\ell+1}$, sauf dans certains cas 
particuliers (voir l'énoncé complet du théorème 
\ref{th.pseudoreg} dans la section  \ref{sec regseq}). 

\paragraph{Nouveaux Null\-stellen\-s\"atze}~

\noindent 
La description de la dimmension de Krull à travers les suites pseudo 
régulières apparaît comme un cas particulier d'un  \nst formel 
plus général démontré dans la section \ref{sec colsim}.

\medskip\noindent
{\bf Théorème \ref{th.nstformel}}
{\it  {\em (\nst formel pour les chaînes d'\idepsz, en \clamaz).} 
Soit $\gA$ un anneau et $J_1,\ldots,J_\ell$,  $U_1, \ldots, U_\ell$  
$2\ell$ parties de $\gA$. Les deux propriétés suivantes sont 
équivalentes{\footnote{Dans tout l'article le symbole
 $\subset$ est mis pour l'inclusion au sens large.}}:
\begin{itemize}
\item [$(a)$] il n'existe pas $\ell$ \ideps $P_1\subset \cdots\subset 
P_\ell$  tels que
 $J_i\subset P_i$, $U_i\cap P_i=\emptyset $, $(i=1,\ldots,\ell)$, 
\item [$(b)$] il existe $j_i$ dans l'idéal engendré par $J_i$, 
$(i=1,\ldots,\ell)$, $u_i$ dans le monoïde engendré par $U_i$, 
$(i=1,\ldots,\ell)$, vérifiant l'égalité  
$$u_1(\ldots(u_\ell+j_\ell)+\ldots)+j_1=0\eqno (*)$$ 
\end{itemize}
}

Ce résultat, énoncé de manière aussi générale, ne peut pas 
être constructif. Nous en donnons deux versions constructives, une 
générale sous forme d'un collapsus simultané
(théorème \ref{th.colsim}), une autre particulière lorsque 
l'anneau $\gA$ vérifie des hypothèses adéquates quant à la 
possibilité d'y mener des calculs sur les \ideps (théorème 
\ref{th.nstformel2}).

\medskip Nous démontrons enfin dans la section \ref{sec icocl} le 
\nst suivant dans le cas géométrique.

\medskip\noindent
{\bf Théorème \ref{th.nst}.}
{\it  Soit $\gK$ un corps contenu dans un corps algébriquement clos 
discret $\gL$. Soient $J_1,\ldots,J_\ell,U_1,\ldots,U_\ell$ $2\ell$ 
familles finies de \pols de $\Kxn=\Kx$. \\
Soient $\cI_i$ l'idéal de $\Kx$ engendré par $J_i$ et $\cM_i$ le 
monoïde engendré par $U_i$, $(i=1,\ldots,\ell)$.\\
On~a alors l'équivalence suivante
\begin{itemize}
\item [$(a)$] Il n'existe pas $\ell$ variétés irréductibles 
$V_\ell\subset \cdots\subset V_1\subset \gL^n$ qui vérifient les 
conditions suivantes
$$\forall f\in J_i\;\;  f|_{V_i}=0|_{V_i},\; \; \; \; \forall g\in 
U_i\; \;  g|_{V_i}\neq 0|_{V_i}\; \;\; \;  (i=1,\ldots,\ell)
$$
\item [$(b)$] Il existe des $u_i\in \cM_i$ et  $j_i\in \cI_i$ 
$(i=1,\ldots,\ell)$ tels que
$$ u_1(\ldots(u_\ell+j_\ell)+\ldots)+j_1= 0
$$
\end{itemize}
En outre on a un procédé explicite qui construit, à partir de la 
donnée  $J_1,\ldots,J_\ell,U_1,\ldots,U_\ell$, ou bien une identité 
algébrique comme en $(b)$ (si une telle identité existe), ou bien 
(dans l'autre cas) un système fini de générateurs pour chacun des 
\ideps correspondant aux variétés irréductibles vérifiant les 
conditions décrites au point $(a)$.
}

\paragraph{La méthode dynamique}~

\noindent 
D'un point de vue dynamique (l'annexe reprend, pour la commodité de 
la lecture, des définitions relatives aux \sads données dans 
\cite{clr} et \cite{lom98}), on accepte de ne pas spécifier 
complètement la chaîne croissante d'\idepsz, et on s'intéresse 
aux conditions dans lesquelles une telle structure collapse.

Un idéal premier est incomplètement spécifié par la donnée 
d'éléments dans l'idéal et d'éléments dans la partie 
multiplicative complémentaire de l'idéal.

Si ensuite on retourne dans une situation concrète où les 
chaînes croissantes d'\ideps sont complètement explicitables, on 
obtient un \nst effectif. C'est le cas de l'anneau $\Lxn$ lorsque $\gL$ 
est un \cacz. 

Une \sad est une {\em structure algébrique incomplètement 
spécifiée}. Pour cela on donne les axiomes que doit vérifier la 
structure, des générateurs et des relations. Le fait que la 
structure est {\em  algébrique} tient à ce que les axiomes ont une 
forme particulièrement simple. Ils peuvent être mis en œuvre à 
l'intérieur d'un calcul arborescent dans lequel la logique sous sa 
forme usuelle sophistiquée ne joue plus aucun role. Aucune formule 
compliquée avec connecteurs ou quantificateurs n'est mise en œuvre. 
Seuls apparaissent les prédicats donnés au départ dans la 
structure. En particulier, on peut éviter tout recours à la 
négation. Et le calcul mis en œuvre est, le long de chaque branche 
du calcul arborescent, de l'algèbre pure, en fait surtout de la 
réécriture. L'arborescence elle-même est gouvernée par des 
axiomes simples. Elle est pour l'essentiel une discussion cas par cas.

\medskip Cet article constitue une confirmation de la souplesse 
d'utilisation des \sads 
(cf.  \cite{clr,KL,lom95,lom97,lom98,lom99,lom99a,lq99}). 
La plupart des théorèmes d'algèbre abstraite, dans la démonstration 
desquels l'axiome du choix et le principe du tiers exclu semblent 
offrir un obstacle sérieux à une interprétation explicite en 
termes de calculs algébriques, auraient ainsi une version dynamique 
constructive, à partir de laquelle l'utilisation du \tcg fournirait 
la version classique abstraite. Lorsque ces théorèmes abstraits 
sont ensuite utilisés pour prouver l'existence d'objets concrets, la 
version dynamique donnera alors ipso facto une construction explicite 
de ces objets concrets.

\section{Théorie dynamique des chaînes croissantes d'\idepsz} 
\label{sec tdychcr}
Nous donnons tout le traitement avec $\ell=3$, laissant au lecteur le 
soin de remplacer éventuellement $3$ par $\ell$ et de rajouter les 
{\em trois petits points} aux endroits nécessaires{\footnote{S'il 
le désire, il pourra même faire des démonstrations par récurrence sur 
$\ell$. Il aura alors la satisfaction secrète de savoir que ses 
énoncés sont prouvés rigoureusement pour tout entier $\ell$ 
même dans le cas où l'on dispose d'entiers non standard.}}.

Lorsqu'on examine des structures algébriques ``au dessus'' d'une 
structure d'anneau commutatif, on peut se passer de la relation 
d'égalité binaire et de la plupart des axiomes des anneaux 
commutatifs. On introduit le prédicat unaire $x=0$ et on considère 
que $x=y$ est une abréviation pour $x-y=0$. Les axiomes usuels des 
anneaux commutatifs sont alors remplacés par du calcul algébrique 
dans les anneaux de \pols à coefficients entiers~: on considère 
que tout terme doit être réécrit sous une forme réduite dans 
l'anneau des \pols correspondant. Par exemple l'axiome $a+b=b+a$ 
est remplacé par le calcul qui réduit 
$(a+b)-(b+a)$  à $0$  dans $\ZZ[a,b]$. Les seuls axiomes qui restent 
nécessaires sont alors les suivants.

\axi{}{0=0}{\Ac_1}   

\axi{(x=0,\; y=0)}{x+y=0}{\Ac_2}   

\axi{x=0}{ax=0}{\Ac_3}

\axio{1=0}{\Faux}{\Ac_{col}}

La \tdy qu'on obtient est dite directe si en outre tous les prédicats 
sont unaires et tous les axiomes sont ``algébriques directs'' en un 
sens convenable (pour plus de détails, voir \cite{clr}). Les \tdys 
directes sont particulièrement bien adaptées aux démonstrations de \nstsz.

Rappelons qu'un \nst est un théorème qui affirme que certains 
systèmes de conditions algébriques sont incompatibles si et 
seulement si il existe un certificat algébrique simple 
d'incompatibilité. Par exemple le \nst faible usuel dit qu'un 
système d'équations polynomiales n'admet aucun zéro (dans une 
cloture algébrique du corps des coefficients) \ssi $1$ est dans 
l'idéal engendré par les équations.

La \tdy des anneaux commutatifs avec trois \ideps emboités est 
basée sur la théorie directe $\cP_{3,dir}$ suivante.

Le langage $\cL_{\cP,3}$ est celui des anneaux commutatifs, avec le 
prédicat unaire $x=0$ et six prédicats unaires $\J1(x)$, $\U1(x)$, 
$\J2(x)$, $\U2(x)$, $\J3(x)$, $\U3(x)$. Si $P_1\subset P_2\subset P_3$ 
sont les trois \ideps que nous avons en vue, $\J{i}(x)$ signifie 
$x\in P_i$ et  $\U{i}(x)$ signifie $x\notin P_i$ ou plutot que $x$ est 
\ndz modulo $P_i$.

Nous donnons les axiomes de cette théorie directe sous forme de trois 
groupes distincts, réservant au troisième groupe, pour des raisons 
de commodité dans l'exposé qui suit, certains axiomes pour les 
parties multiplicatives.

Le premier groupe d'axiomes, noté $(\Ax_0)$ est formé par les 4 
axiomes des anneaux commutatifs non triviaux.

\axi{}{0=0}{\Ac_1}   

\axi{(x=0,\; y=0)}{x+y=0}{\Ac_2}   

\axi{x=0}{ax=0}{\Ac_3}

\axio{1=0}{\Faux}{\Ac_{col}}

Le deuxième groupe d'axiomes, noté $(\Ax_1)$ est le suivant. Un 
axiome indiqué avec un indice $i$ est répété pour $i=1,2,3$, 
sauf si $i+1$ intervient également dans l'axiome, auquel cas l'axiome  
est répété pour $i=1,2$.

\axi{(x=0)}{\J{i}(x)}{\rP_{3,dir,1,i}}

\axi{(\J{i}(x),\; \J{i}(y))}{\J{i}(x+y)}{\rP_{3,dir,2,i}}

\axi{\J{i}(x)}{\J{i}(ax)}{\rP_{3,dir,3,i}}

\axi{\J{i}(x)}{\J{i+1}(x)}{\rP_{3,dir,4,i}}

\axi{\J{i}(1)}{\Faux}{\rP_{3,col,1,i}}

\axi{}{\U{i}(1)}{\rP_{3,dir,5,i}}

\axi{(\U{i}(x),\; \U{i}(y))}{\U{i}(xy)}{\rP_{3,dir,6,i}}

\axio{(x=0,\; \U{i}(y))}{\U{i}(x+y)}{\rP_{3,dir,7,i}}

\noi Le troisième groupe d'axiomes, noté $(\Ax_2)$ est le suivant.

\axi{\U{i}(0)}{\Faux}{\rP_{3,col,2,i}}

\axi{(\J{i}(x),\; \U{i}(y))}{\U{i}(x+y)}{\rP_{3,dir,8,i}}

\axi{\U{i+1}(x)}{\U{i}(x)}{\rP_{3,dir,9,i}}

\begin{notation} 
\label{nota.tdir}
{\rm  Nous notons $\cP_{3,dir}$ la théorie directe $(\cL_{\cP_3}\vert 
\Ax_0,\Ax_1,\Ax_2)$.
} 
\end{notation}
On~a le lemme suivant qui sera utilisé implicitement dans la suite.
\begin{lemma} 
\label{lem.col0} Pour une \sad de type $\cP_{3,dir}$ si on a établi 
deux faits $\U{i}(t)$ et $\J{i}(t)$ pour un même terme $t$, cela 
collapse.
\end{lemma}
\begin{proof}{Démonstration.} On~a en effet $\J{i}(-t)$ par $(\rP_{3,dir,3,i})$ 
donc $\U{i}(t+(-t))$ par $(\rP_{3,dir,8,i})$ c.-à-d. $\U{i}(0)$
\end{proof}
Pour obtenir la \tdy correspondant vraiment à un sytème de trois 
\ideps emboités nous devons rajouter le groupe d'axiomes suivant, 
noté $(\Ax_5)$.

\axio{}{(\J{i}(x)\; \ou\;  \U{i}(y))}{\rP_{3,dyn,1,i}}

Si enfin nous réclamons que l'emboitement des idéaux soit strict, 
nous rajouterons le groupe d'axiomes suivant, noté $(\Ax_6)$.  

\axio{}{\E x_i\;\;   (\J{i+1}(x_i), \; \U{i}(x_i)) }{\rP_{3,dyn,2,i}}

\begin{notation} 
\label{nota.tdyn}
{\rm  Nous noterons \\
$\cP_{3,dyn1}=(\cL_{\cP_3}\vert \Ax_0,\Ax_1,\Ax_2,\Ax_5)$ et  
$\cP_{3,dyn2}=(\cL_{\cP_3}\vert \Ax_0,\Ax_1,\Ax_2,\Ax_5,\Ax_6)$.
} 
\end{notation}

En pratique, on sera plutot intéressé par des \sads de type 
$\cP_{3,dyn1}$ avec une présentation pour laquelle on peut trouver 
deux termes $u$ et $v$ tels que  $\J2(u), \U1(u), \J3(v), \U2(v)$ 
soient des faits prouvables.

\begin{lemma} 
\label{lem.rsimp}
La théorie $\cP_{3,dyn1}$ prouve les {\em règles de simplification} 
suivantes~: tout d'abord un groupe d'axiomes  $(\Ax_3)$

\axio{(\J{i}(xy),\;  \U{i}(y))}{\J{i}(x)}{\rP_{3,sim,1,i}}

\noi ensuite un groupe d'axiomes  $(\Ax_4)$

\axi{\U{i}(xy)}{(\U{i}(x),\; \U{i}(y))}{\rP_{3,sim,2,i}}

\axi{\J{i}(x^2)}{\J{i}(x)}{\rP_{3,sim,3,i}}
\end{lemma}
\begin{proof}{Démonstration.}
Par exemple pour la première règle. On suppose $\J1(xy),\;  \U1(y)$ 
on ouvre deux branches, la première avec $\J1(x)$ et la deuxième 
avec $\U1(x)$. Dans la deuxième on a $\U1(xy)$ et $\J1(-xy)$ donc 
$\U1(xy+(-xy))$ c.-à-d. $\U1(0)$ et elle meurt.
\end{proof}
\begin{notation} 
\label{nota.tsim}
{\rm  Nous noterons $\cP_{3,sim1}=(\cL_{\cP_3}\vert 
\Ax_0,\Ax_1,\Ax_2,\Ax_3,\Ax_4)$ et  $\cP_{3,sim2}=(\cL_{\cP_3}\vert 
\Ax_0,\Ax_1,\Ax_3)$. 
} 
\end{notation}
Une structure algébrique ordinaire qui vérifie les axiomes de 
$\cP_{3,sim2}$ est donnée par un anneau commutatif, trois idéaux 
emboités $I_1\subset I_2\subset I_3$ et  trois monoïdes 
multiplicatifs $\; S_1,\; S_2,\; S_3$. Les éléments de $S_i$ 
doivent être non diviseurs de zéro modulo $I_i$. Un exemple est 
donné en prenant pour $S_i$ tous les non diviseurs de zéro modulo 
$I_i$. Par contre on ne demande pas d'inclusion entre les $S_i$, et les 
idéaux ne sont pas nécessairement premiers.
\section{Collapsus simultanés et Null\-stellen\-s\"atze formels} 
\label{sec colsim}

Dans une \tdy comme $\cP_{3,dir}$, on peut donner une description 
simplifiée pour la présentation d'une \sad $S$. L'ensemble des 
générateurs est $G$ et les relations de départ peuvent être 
données par 7 parties de $\Zg$, $\Rzero$, $\Rj{i}$,  $\Ru{i}$  
$(i=1,2,3)$. La partie  $\Rzero$ donne l'ensemble des termes $t\in\Zg$ 
pour lesquels on a la relation de départ $t=0$, la partie  $\Rj1$ 
donne l'ensemble des termes $t\in\Zg$ pour lesquels on a la relation de 
départ $\J1(t)$, etc...
Le lemme suivant est à la fois très simple et fondamental.

\begin{lemma} 
\label{lem.coll} Soit $S=(\cP_{3,dir},(G\vert R))$ avec 
$R=(\Rzero,\Rj1,\ldots,\Ru3)$ une \sad de type $\cP_{3,dir}$. 
Notons $\Mu{i}$ le monoide (multiplicatif) dans $\Zg$ engendré par 
$\Ru{i}$ $(i=1,2,3)$. 
Notons  $\Izero$ l'idéal de $\Zg$ engendré par $\Rzero$ et 
$\Ij{i}$ l'idéal de $\Zg$ engendré par $\Rj{i}$  $(i=1,2,3)$.\\
Alors la \sad $S$ collapse si et seulement si on a une identité 
algébrique du type suivant dans $\Zg$
$$u_1(u_2(u_3+j_3)+j_2)+j_1+j_0\egal u_1u_2u_3+u_1u_2j_3+u_1j_2+j_1+j_0
\egal 0
$$
avec $u_i\in\Mu{i}$ et $j_i\in\Ij{i}$ pour $(i=1,2,3)$ et $j_0\in 
\Izero$.
\end{lemma}
\begin{proof}{Démonstration.} On voit tout d'abord facilement que l'on peut 
supposer que le seul axiome de collapsus est $\; \U1(0)\; \im\; \Faux$. 
On considère alors les faits prouvables par \evd avant un 
collapsus.\\
Notons $\cI_0=\Izero$, $\cI_1=\Izero+\Ij1$, $\cI_2=\Izero+\Ij1+\Ij2$, 
$\cI_3=\Izero+\Ij1+\Ij2+\Ij3$.\\
Vu l'enchainement des axiomes directs, on voit alors que  les 
éléments $t\in\Zg$ pour lesquels le fait $t=0$ est prouvable sont 
exactement les éléments de $\cI_0=\Izero$. Ensuite les éléments 
$t\in\Zg$ pour lesquels le fait $\J1(t)$ est prouvable sont exactement 
les éléments de $\cI_1$. Même chose pour $\J2(t)$ et $\cI_2$, 
puis pour $\J3$ et $\cI_3$.\\  
Ensuite les éléments $t\in\Zg$ pour lesquels le fait $\U3$  est 
prouvable sont exactement les éléments de la forme $u_3+k_3$ avec 
$u_3\in\Mu3$ et $k_3\in\cI_3$ (notez que le produit de deux 
éléments de cette forme est encore un élément de cette forme). 
Puis les éléments $t\in\Zg$ pour lesquels le fait $\U2$  est 
prouvable sont exactement les éléments de la forme 
$u_2(u_3+k_3)+k_2$ avec $u_3\in\Mu3$, $u_2\in\Mu2$, $k_3\in\cI_3$ et 
$k_2\in\cI_2$ (notez à nouveau que le produit de deux éléments de 
cette forme est encore un élément de cette forme). Enfin les 
éléments $t\in\Zg$ pour lesquels le fait $\U1$  est prouvable sont 
exactement les éléments de la forme $u_1(u_2(u_3+k_3)+k_2)+k_1$ 
avec $u_3\in\Mu3$, $u_2\in\Mu2$, $u_1\in\Mu1$, $j_3\in\cI_3$, 
$j_2\in\cI_2$ et $j_1\in\cI_1$.\\
On termine en remarquant que les éléments 
$u_1(u_2(u_3+k_3)+k_2)+k_1$ de la forme précédente sont exactement 
les éléments de la forme $u_1(u_2(u_3+j_3)+j_2)+j_1+j_0$ avec 
$u_i\in\Mu{i}$ et $j_i\in\Ij{i}$ pour $(i=1,2,3)$ et $j_0\in \Izero$.
\end{proof}

Nous sommes maintenant en mesure d'établir un théorème de 
collapsus simultané.

\begin{theorem} {\em (théorème de collapsus simultané).}
\label{th.colsim} Les \tdys $\cP_{3,dir}$, $\cP_{3,dyn1}$, 
$\cP_{3,sim1}$ et $\cP_{3,sim2}$ collapsent simultanément. Les \tdys  
$\cP_{3,dyn1}$ et $\cP_{3,sim1}$ prouvent les mêmes faits. 
\end{theorem}

Avant de prouver ce théorème, donnons une version classique 
(obtenue en supposant le \tcgz) pour le collapsus simultané de 
$\cP_{3,dir}$ et $\cP_{3,dyn1}$.

\begin{theorem} 
\label{th.nstformel} {\em (\nst formel pour les chaînes d'\idepsz, 
en \clamaz).} Soit $\gA$ un anneau et $J_1$, $J_2$, $J_3$,  $U_1$, $U_2$, 
$U_3$  six parties de $\gA$. Les deux propriétés suivantes sont 
équivalentes~:
\begin{itemize}
\item [$(a)$] il n'existe pas 3 \ideps $P_1\subset P_2\subset P_3$ tels 
que $J_1\subset P_1$, $U_1\cap P_1=\emptyset $, $J_2\subset P_2$, 
$U_2\cap P_2=\emptyset $, $J_3\subset P_3$, $U_3\cap P_3=\emptyset$
\item [$(b)$] il existe $j_1$ dans l'idéal engendré par $J_1$, 
$j_2$ dans l'idéal engendré par  $J_2$, $j_3$ dans l'idéal 
engendré par  $J_3$, $u_1$ dans le monoïde engendré par $U_1$, 
$u_2$ dans le monoïde engendré par $U_2$, $u_3$ dans le 
monoïde engendré par $U_3$, vérifiant l'égalité  
$$u_1(u_2(u_3+j_3)+j_2)+j_1 = 0. \quad (*)$$ 
\end{itemize}
\end{theorem}

\begin{proof}{Démonstration du théorème \ref{th.nstformel}.} Il est clair 
qu'une égalité $(*)$ interdit la possibilité des trois \idepsz.  
Si par contre la \sad de type $\cP_{3,dyn1}$ correspondant aux 
données  $\gA$, $J_1$, $J_2$, $J_3$, $U_1$, $U_2$ et $U_3$ ne collapse 
pas (c.-à-d. s'il n'y a pas d'égalité du type~$(*)$), alors la 
\tfp correspondante est cohérente (cf. \cite{clr} théorème 1.1) 
et par le \tcg elle admet un modèle. Ce modèle est donné par un 
anneau $\gB$ avec trois \ideps  $Q_1\subset Q_2\subset Q_3$, avec un 
\homo $\varphi :\gA\rightarrow \gB$ vérifiant $\varphi(J_1)\subset Q_1$, 
$\varphi(U_1)\cap Q_1=\emptyset $, $\varphi(J_2)\subset Q_2$, 
$\varphi(U_2)\cap Q_2=\emptyset $, $\varphi(J_3)\subset Q_3$, 
$\varphi(U_3)\cap Q_3=\emptyset$. Il suffit alors de prendre pour $P_i$ 
les $\varphi ^{-1}(Q_i)$.\\
Une démonstration classique plus usuelle serait la suivante.\\ Considérons 
les systèmes $(J'_1, J'_2, J'_3, U'_1, U'_2, U'_3)$ où les $J'_i$ 
sont des idéaux de $\gA$, les $U'_i$ sont des monoïdes, avec 
$J_{i}\subset J'_{i}$, $U_{i}\subset U'_{i}$, $J'_1\subset J'_2\subset 
J'_3$, $U'_3\subset U'_2\subset U'_1$, avec la condition 
supplémentaire qu'une égalité de type $(*)$ est impossible. 
Ordonnons l'ensemble de ces systèmes par l'ordre produit de l'ordre 
de l'inclusion. On est dans les conditions d'application du lemme de 
Zorn. On considère un système maximal $(J'_1, J'_2, J'_3, U'_1, 
U'_2, U'_3)$. Si on n'a pas $J'_1\cup U'_1=\gA$ soit $t\in 
\gA\setminus(J'_1\cup U'_1)$. Si on rajoute $t$ à $J'_1$, puisque le 
système est maximal, on a une égalité  
$u_1(u_2(u_3+j_3)+j_2)+j_1+at=0$. Si on rajoute $t$ à $U'_1$, puisque 
le système est maximal, on a une égalité  
$t^mu_1(u_2(u_3+j_3)+j_2)+j_1=0$. Ici on fait un calcul algébrique 
qui se trouve dans la démonstration du lemme \ref{lem.colsim} ci-après, pour 
obtenir une contradiction. Même chose dans les cas où $J'_2\cup 
U'_2\neq \gA$ ou $J'_3\cup U'_3\neq \gA$.
\end{proof}
La démonstration précédente montre qu'un contenu constructif du 
théorème \ref{th.nstformel} est exactement le théorème 
\ref{th.colsim}. Plus exactement,  le théorème \ref{th.colsim} et 
le lemme  \ref{lem.coll} (qui sont tous deux constructifs) donnent la 
\cnes de collapsus pour la \tdy $\cP_{3,dyn1}$ donc aussi pour la \tfp 
correspondante. Or, via le \tcgz, cette condition de collapsus 
équivaut facilement au théorème \ref{th.nstformel}.
\begin{proof}{Démonstration du théorème \ref{th.colsim}.} Nous reprenons 
concernant la présentation $(G\vert R)$ d'une \sad de type 
$\cP_{3,dir}$ les notations du lemme \ref{lem.coll}. 
Supposons tout d'abord que nous avons le collapsus simultané de 
$\cP_{3,dir}$ et $\cP_{3,dyn1}$, alors nous en déduisons celui de  
$\cP_{3,sim1}$ qui est une théorie intermédiare entre ces deux 
extrêmes. Par ailleurs $\cP_{3,sim2}$ est plus faible que 
$\cP_{3,sim1}$. Il suffit donc de montrer qu'une égalité 
$u_1(u_2(u_3+j_3)+j_2)+j_1\egal0$
avec $u_i\in\Mu{i}$ et $j_i\in\Ij{i}$ pour $(i=1,2,3)$ fait collapser 
la \sad $S'$ de type $\cP_{3,sim2}$ et de présentation $(G\vert R)$. 
Tout d'abord, il est clair qu'on a les faits  $\U1(u_1)$, $\U2(u_2)$, 
$\U3(u_3)$, $\J1(j_1)$, $\J2(j_2)$, $\J3(j_3)$. Posons 
$v=u_2(u_3+j_3)+j_2$. Comme $u_1v=-j_1$ la règle de simplification 
$(\rP_{3,sim,1,1})$ nous permet de conclure $\J1(v)$. Donc on a 
$\J2(v)$ c.-à-d. $\J2(u_2(u_3+j_3)+j_2)$. Comme $\J2(-j_2)$ on en 
déduit  $\J2(u_2(u_3+j_3))$.  La règle de simplification 
$(\rP_{3,sim,1,2})$ nous permet de conclure $\J2(u_3+j_3)$ donc 
$\J3(u_3+j_3)$. On en déduit  $\J3(u_3)$, comme on a aussi 
$\U3(u_3)$, la régle de simplification $(\rP_{3,sim,1,3})$ donne 
$\J3(1)$ et cela collapse.

\sni Montrons maintenant le  collapsus simultané de $\cP_{3,dir}$ et 
$\cP_{3,dyn1}$. Pour cela il nous suffit de voir qu'un usage une fois 
de l'un des trois axiomes supplémentaires $(\rP_{3,dyn,1,i})$ ne 
change pas les conditions du collapsus. 
Cela signifie qu'il nous faut prouver le lemme suivant~:
\begin{lemma} 
\label{lem.colsim}
Soit $S=(\cP_{3,dir},(G\vert R))$ une \sad (notations  du lemme 
\ref{lem.coll}). Soit $t$ un terme de $\Zg$. Notons $R\cup \{ \J1(t)\}$ 
le système de relations obtenu en rajoutant $\J1(t)$, c.-à-d. en 
rempla\c{c}ant $\Rj1$ par $\Rj1\cup \{t \}$. Notons en outre $S\cup 
(\emptyset\vert \{ \J1(t)\})$ la \sad $(\cP_{3,dir},(G\vert R\cup \{ 
\J1(t)\}))$ 
\begin{itemize}
\item [---] Si les structures $S\cup (\emptyset \vert \{ \J1(t)\})$ et 
$S\cup (\emptyset \vert \{ \U1(t)\})$ collapsent  alors la strucure $S$ 
collapse.
\item [---] Si les structures $S\cup (\emptyset \vert \{ \J2(t)\})$ et 
$S\cup (\emptyset \vert \{ \U2(t)\})$ collapsent  alors la strucure $S$ 
collapse.
\item [---] Si les structures $S\cup (\emptyset \vert \{ \J3(t)\})$ et 
$S\cup (\emptyset \vert \{ \U3(t)\})$ collapsent  alors la strucure $S$ 
collapse.
\end{itemize}
\end{lemma}

\begin{proof}{Démonstration du lemme \ref{lem.colsim}.}~\\
--- La structure $S\cup (\emptyset\vert \{ \J1(t)\})$ collapse \ssi on 
a une égalité dans $\Zg$ 
$$u_1(u_2(u_3+j_3)+j_2)+j_1+at+j_0\egal0 \eqno (1)$$
avec $a\in\Zg$, $u_i\in\Mu{i}$, $j_0\in\Izero$ et $j_i\in\Ij{i}$ pour 
$(i=1,2,3)$.
La structure $S\cup (\emptyset\vert \{ \U1(t)\})$ collapse \ssi on a 
une égalité dans $\Zg$ 
$$t^mv_1(v_2(v_3+k_3)+k_2)+k_1+k_0\egal0 \eqno (2)$$
avec $m\in \NN$, $v_i\in\Mu{i}$, $k_0\in\Izero$ et $k_i\in\Ij{i}$ pour 
$(i=1,2,3)$.
On peut terminer par l'usuel truc de Rabinovitch. Dans $(1)$ on met 
$-at$ dans le second membre, on élève à la puissance $m$ on 
obtient une égalité dans $\Zg$ qui se réécrit sous la forme
$$u'_1(u'_2(u'_3+j'_3)+j'_2)+j'_1+j'_0\egal bt^m \eqno (3)$$
avec $b\in\Zg$, $u'_i\in\Mu{i}$, $j'_0\in\Izero$ et $j'_i\in\Ij{i}$ 
pour $(i=1,2,3)$. Il reste à multiplier $(2)$ par $b$ puis à 
remplacer $bt^m$ par sa valeur donnée en $(3)$. On obtient
$$(u'_1(u'_2(u'_3+j'_3)+j'_2)+j'_1+j'_0)v_1(v_2(v_3+k_3)+k_2) +bk_1+bk_0
\egal0 \eqno (4)$$
qui se réécrit sous la forme voulue.

\sni --- On reprend le même style de notations, sans forcément 
donner toutes les précisions. La structure $S\cup (\emptyset\vert \{ 
\J2(t)\})$ collapse \ssi on a une égalité dans $\Zg$ 
$$u_1(u_2(u_3+j_3)+j_2+at)+j_1+j_0\egal0 \eqno (5)$$
La structure $S\cup (\emptyset\vert \{ \U2(t)\})$ collapse \ssi on a 
une égalité dans $\Zg$ 
$$v_1(t^mv_2(v_3+k_3)+k_2)+k_1+k_0\egal0 \eqno (6)$$
On réécrit $(5)$ sous la forme $(7)$
$$u_1(u_2(u_3+j_3)+j_2)+j_1+j_0\egal -u_1at \eqno (7)$$
En élevant ceci à la puissance $m$ on obtient
$$u'_1(u'_2(u'_3+j'_3)+j'_2)+j'_1+j'_0\egal u'_1bt^m \eqno (8)$$
On multiplie $(6)$ par $u'_1b$ on obtient
$$v_1(u'_1bt^mv_2(v_3+k_3)+u'_1k'_2)+k'_1+k'_0 \egal 0 \eqno (9)$$
En utilisant $(8)$ cela donne
$$v_1((u'_1(u'_2(u'_3+j'_3)+j'_2)+j'_1)v_2(v_3+k'_3)+u'_1k'_2)+ 
k'_1+k'_0\egal0 \eqno (10)$$
ce qui se réécrit
$$v_1u'_1((u'_2(u'_3+j'_3)+j'_2)v_2(v_3+k'_3)+k'_2)+j''_1+j''_0\egal 0 
\eqno (11)$$
enfin $(u'_2(u'_3+j'_3)+j'_2)v_2(v_3+k'_3)+k'_2$ se réécrit
$u_2''(u_3''+j''_3)+j''_2.$

\sni --- Le troisième cas est essentiellement le même que le 
second.
\end{proof}

\sni Montrons enfin que $\cP_{3,dyn1}$ et $\cP_{3,sim1}$ prouvent les 
mêmes faits. Comme $\cP_{3,sim1}$  est plus faible que $\cP_{3,dyn1}$ 
et comme les deux théories collapsent simultanément, il nous suffit 
de vérifier le lemme suivant.
\begin{lemma} 
\label{lem.faitcol} Soit $S$ une \sad de type $\cP_{3,sim1}$ et $t$ un 
élément de $S$ (plus précisément $t$ est un élément de 
$\Zg$ si $G$ est l'ensemble générateur de $S$). Alors
\begin{itemize}
\item [---] un fait $\J{i}(t)$ est vrai dans $S$ \ssi rajouter 
$\U{i}(t)$ aux relations de $S$ produit un collapsus,
\item [---] un fait $\U{i}(t)$ est vrai dans $S$ \ssi rajouter 
$\J{i}(t)$ aux relations de $S$ produit un collapsus.
\end{itemize}
\end{lemma}
Naturellement ce résultat vaut aussi pour $\cP_{3,dyn1}$.
\begin{proof}{Démonstration du lemme \ref{lem.faitcol}.}
Supposons par exemple que la structure $S\cup (\emptyset\vert \{ 
\J2(t)\})$ collapse~: on a une égalité dans $\Zg$ 
$$u_1(u_2(u_3+j_3)+j_2+at)+j_1+j_0\egal0 \eqno (5)$$
On~a donc $\J1(u_1(u_2(u_3+j_3)+j_2+at))$ et en appliquant 
$(\rP_{3,sim,1,1})$ on obtient  $\J1(u_2(u_3+j_3)+j_2+at)$ donc 
$\J2(u_2(u_3+j_3)+j_2+at)$ donc $\J2(u_2(u_3+j_3)+at)$. Posons 
$j=u_2(u_3+j_3)+at$.   
Comme on a $-at=-j+(u_2(u_3+j_3))$, $\J2(-j)$ et $\U2(u_2(u_3+j_3))$, on 
obtient $\U2(-at)$. Et finalement $\U2(t)$ en appliquant 
$(\rP_{3,sim,2,2})$.\\
Supposons maintenant par exemple que la structure $S\cup 
(\emptyset\vert \{ \U2(t)\})$ collapse~: on a une égalité dans 
$\Zg$ 
$$v_1(t^mv_2(v_3+k_3)+k_2)+k_1+k_0\egal0 \eqno (6)$$
Comme précédemment on obtient $\J2(t^mv_2(v_3+k_3)+k_2)$, puis par 
$(\rP_{3,sim,1,2})$, $\J2(t^m)$, donc aussi $\J2(t^{2^\ell})$ pour un 
certain $\ell$. En appliquant $\ell$ fois $(\rP_{3,sim,3,2})$ il vient 
$\J2(t)$.
\end{proof}
\end{proof}

Revenons maintenant au \nst formel (théorème \ref{th.nstformel}), 
pour la démonstration duquel on a besoin de l'axiome du choix{\footnote{~Le 
\tcg est, du point de vue classique, équivalent à une forme faible 
de l'axiome du choix. Dans le cas d'une théorie dynamique avec un 
langage dénombrable et dont les axiomes peuvent être 
énumérés, le \tcg  est, du point de vue constructif,  
équivalent au principe d'omniscience {\bf  LLPO} qui affirme que tout 
nombre réel est $\leq 0$ ou $\geq 0$. Ceci dans la mesure où on 
considère comme non problématique l'axiome du choix 
dénombrable.}}. Il existe des cas où ce théorème est vrai d'un 
point de vue constructif.  

Pour la définition d'un anneau de Lasker-Noether nous renvoyons à
\cite{MRR}. En \clama tout anneau noethérien est un anneau de
Lasker-Noether. 
\begin{theorem} 
\label{th.nstformel2} {\em (\nst formel pour les chaînes 
d'\idepsz, en \maths constructives).} Soit $\gA$ un anneau de 
Lasker-Noether. Soient $J_1$, $J_2$, $J_3$,  $U_1$, $U_2$, $U_3$  six 
parties finies de $\gA$. Les deux propriétés suivantes sont 
équivalentes~:
\begin{itemize}
\item [$(a)$] il n'existe pas 3 \ideps $P_1\subset P_2\subset P_3$  de 
type fini   tels que $J_1\subset P_1$, $U_1\cap P_1=\emptyset $, 
$J_2\subset P_2$, 
$U_2\cap P_2=\emptyset $, $J_3\subset P_3$, $U_3\cap P_3=\emptyset$
\item [$(b)$] il existe $j_1$ dans l'idéal engendré par $J_1$, 
$j_2$ dans l'idéal engendré par  $J_2$, $j_3$ dans l'idéal 
engendré par  $J_3$, $u_1$ dans le monoïde engendré par $U_1$, 
$u_2$ dans le monoïde engendré par $U_2$, $u_3$ dans le 
monoïde engendré par $U_3$, vérifiant l'égalité  
$$u_1(u_2(u_3+j_3)+j_2)+j_1=0\eqno (*)$$ 
\end{itemize}

\noi Plus précisément, il existe un algorithme qui construit ou 
bien une égalité du type $(*)$ en $(b)$ ou bien trois \ideps 
vérifiant les conditions requises en $(a)$. 
\end{theorem}
\begin{proof}{Démonstration.} Notons $I_1=(J_1)$, $I_2=(J_1)+(J_2)$ et 
$I_3=(J_1)+(J_2)+(J_3)$. Notons $V_3$  la partie multiplicative 
engendrée par $U_3$, $V_2$ la partie multiplicative engendrée par 
$U_3$ et $U_2$, $V_1$ la partie multiplicative engendrée par $U_3$, 
$U_2$ et $U_1$.\\ 
Dans l'anneau $\gA$ il est possible de construire explicitement la 
famille finie des \ideps minimaux au dessus d'un idéal de type fini 
donné. En outre si $P$ est un idéal premier de type fini explicite 
et $U$ une partie multiplicative de type fini explicite, il est 
possible de tester si $P\cap U=\emptyset $. En effet, cela équivaut 
à $u\notin P$ pour chacun des générateurs de $U$. \\
Si trois \ideps $P_1$, $P_2$, $P_3$ vérifiant les conditions requises 
existent, alors on peut tout d'abord remplacer $P_1$ par un idéal 
premier $Q_1$ minimal au dessus de $I_1$ et contenant~$P_1$. Ensuite on 
peut remplacer $P_2$ par un idéal premier $Q_2$ minimal au dessus de 
$I_2+Q_1$ et contenant~$P_2$. Enfin,  on peut remplacer $P_3$ par un 
idéal premier $Q_3$ minimal au dessus de $I_3+Q_2$ et contenant~$P_3$.\\
Partant de $I_1$, $I_2$, $I_3$ il n'y a qu'un nombre fini de 
possibilités, qu'on peut toutes expliciter, pour trois idéaux 
$Q_1$, $Q_2$, $Q_3$ vérifiant~: $Q_1$ est minimal au dessus de $I_1$, 
$Q_2$ est minimal au dessus de $I_2+Q_1$ et  $Q_3$ est minimal au 
dessus de $I_3+Q_2$. Pour chaque triplet $(Q_1,Q_2,Q_3)$ ainsi 
sélectionné, il reste à tester si $Q_1\cap V_1=Q_2\cap 
V_2=Q_3\cap V_3=\emptyset$.\\
Nous venons de voir que l'existence éventuelle de trois \ideps 
vérifiant les conditions requises peut être testée explicitement, 
avec construction d'un triplet convenable en cas de réponse 
positive.\\
Il nous reste à voir qu'en cas de réponse négative nous avons un 
collapsus de la \sad correspondante, ce qui d'après le lemme 
\ref{lem.coll} et le théorème de collapsus simultané 
(théorème \ref{th.colsim}), donnera une égalité du type $(*)$. 
Il s'agit de la \sad $S=(\cP_{3,dyn1},(G\vert R))$ avec $G=\gA$, 
$R=(\Rzero,\Rj1,\ldots,\Ru3)=(\Rzero,J_1,J_2,J_3,U_1,U_2,U_3)$ où~$\Rzero$ est le {\em  diagramme de l'anneau $\gA$}{\footnote{$\Rzero$ 
contient les éléments suivants de $\Zg$~: $1_\Zg-1_\gA$, $0_\Zg-0_\gA$ 
ainsi que tous les $a+_\Zg b-_\Zg c$ et les $a\times_\Zg b-_\Zg d$ 
lorsque $a+_\gA b=_\gA c$ et $a\times_\gA  b=_\gA d$.}}.\\    
Considérons les \ideps minimaux $Q_{1,\ell}$ ($\ell=1,\ldots,n_1$) au 
dessus de $I_1$, que nous avons de manière explicite. Nous savons que 
le radical de $I_1$ est l'intersection de ces \idepsz. Supposons tout 
d'abord que la construction soit bloquée ici par le fait que $U_1\cup 
U_2\cup U_3$ coupe chacun des  $Q_{1,\ell}$. En faisant le produit des 
éléments en question, on obtient de manière explicite un 
élément $v$ de $V_1\cap\sqrt{I_1}$. Puisque les démonstrations dans 
\cite{MRR} sont constructives, nous connaissons un exposant $m$ tel que 
$v^m$ est dans $I_1$ de manière explicite. Ceci produit un collapsus 
de la structure $S$.\\
Supposons maintenant par exemple que $U_1\cup U_2\cup U_3$ (donc aussi 
$V_1$) coupe tous les $Q_{1,\ell}$  sauf~$Q_{1,1}$,~$Q_{1,2}$ et~$Q_{1,3}$.  
Nous allons voir que nous pouvons construire une \evd de $S$ dans 
laquelle toutes les branches meurent sauf certaines réparties en 
trois  types. Dans les branches du premier type sont vrais les 
$\J1(a_i)$ pour un nombre fini de $a_i$ qui engendrent  $Q_{1,1}$.   
Dans les branches du second type sont vrais les $\J1(b_j)$ pour un 
nombre fini de~$b_j$ qui engendrent  $Q_{1,2}$.   Dans les branches du  
troisième type sont vrais les $\J1(c_k)$ pour pour un nombre fini de 
$c_k$ qui engendrent  $Q_{1,3}$.\\
Pour ce faire nous procédons comme suit~: nous choisissons pour 
chaque $\ell\neq 1,2,3$ un élément dans $Q_{1,\ell}\cap (U_1\cup 
U_2\cup U_3)$, nous appelons $u$ leur produit. On~a évidemment 
$\U1(u)$. Si maintenant nous choisissons un élément $a$ de 
$Q_{1,1}$, un élément  $b$ de $Q_{1,2}$ et  un élément $c$ de 
$Q_{1,3}$. On~a $uabc$ dans le radical de $I_1$. D'où une égalité 
explicite entre un élément de $I_1$ et $(uabc)^m$ pour un certain 
$m$. On~a donc $\J1((uabc)^m)$ et $\U1(u)$ d'où on tire $\J1(abc)$. 
\\
Après avoir sélectionné $n_1$ éléments $a_i$ qui engendrent  
$Q_{1,1}$, $n_2$ éléments $b_j$ qui engendrent~$Q_{1,2}$ et~$n_3$ 
éléments  $c_k$ qui engendrent  $Q_{1,3}$, on applique 
systématiquement les axiomes \hbox{$\J1(x) \lor \U1(x)$}~à ces 
éléments, ce qui donne $2^{n_1+n_2+n_3}$ branches. Toute branche 
où l'on~a pour un triplet $(a_i,b_j,c_k)$, $\U1(a_i)$, $\U1(b_j)$ et 
$\U1(c_k)$ meurt puisqu'on~a aussi $\J1(a_ib_jc_k)$. Toute branche qui 
reste en vie est donc de l'un des trois types annoncé.\\
Dans les branches du premier type, nous considérons la liste finie 
des \ideps minimaux au dessus de $I_2+Q_{1,1}$. Supposons par exemple 
que chacun de ces \ideps coupe $U_2\cup U_3$ (donc aussi $V_2$). En 
considérant un produit des éléments en question, on tombe 
explicitement dans le radical de $I_2+Q_{1,1}$ et cela fait mourir la 
branche. etc... \\
En définitive l'obstruction à la construction des \ideps 
vérifiant les conditions requises produit bien un collapsus de la 
\sad $S$.
\end{proof}

\section{Un Null\-stellen\-satz lié à la théorie de la dimension} 
\label{sec icocl}
\begin{theorem} 
\label{th.nst} Soit $\gK$ un corps contenu dans un corps 
algébriquement clos discret $\gL$. Soient $J_1,J_2,J_3,U_1,U_2,U_3$ 
six familles finies de \pols de $\Kxn=\Kx$. \\
Soient $\cI_i$ l'idéal de $\Kx$ engendré par $J_i$ et $\cM_i$ le 
monoïde engendré par $U_i$, $(i=1,2,3)$.\\
On~a alors l'équivalence suivante
\begin{itemize}
\item [$(a)$] Il n'existe pas trois variétés irréductibles 
$V_3\subset V_2\subset V_1\subset \gL^n$ qui vérifient les conditions 
suivantes
$$\forall f\in J_i\;\;  f|_{V_i}=0|_{V_i},\; \; \; \; \forall g\in 
U_i\; \;  g|_{V_i}\neq 0|_{V_i}\; \;\; \;  (i=1,2,3)
$$
\item [$(b)$] Il existe des $u_i\in \cM_i$ et  $j_i\in \cI_i$ 
$(i=1,2,3)$ tels que
$$ u_1(u_2(u_3+j_3)+j_2)+j_1= 0
$$
\end{itemize}
En outre, on a un procédé explicite qui construit, à partir de la 
donnée  $J_1,\ldots,J_3,U_1,\ldots,U_3$, ou bien une identité 
algébrique comme en $(b)$ (si une telle identité existe), ou bien 
(dans l'autre cas) un système fini de générateurs pour chacun des 
\ideps correspondant aux variétés irréductibles vérifiant les 
conditions décrites au point $(a)$.
\end{theorem}
\begin{proof}{Démonstration  (constructive).} Si $\gK=\gL$ c'est une 
conséquence immédiate du \nst formel donné au théorème 
\ref{th.nstformel2}. En effet on sait que l'anneau $\Lxn$ est un anneau 
de Lasker-Noether (cf. \cite{MRR} chap. 8 th. 9.6).\\ 
Si $\gK$ n'est pas algébriquement clos, on se ramène au cas 
précédent en vérifiant que la condition~$(b)$ vue dans $\Lx$ 
équivaut à la condition $(b)$ vue dans $\Kx$. Considérons la \sad 
$S$ définie dans la démonstration du théorème \ref{th.nstformel2}. Nous 
rajoutons les axiomes de cloture algébrique relatifs à $\gK$. Nous 
obtenons une \sad $S'$. Pour voir qu'un collapsus de $S'$ produit un 
collapsus de $S$, il suffit de considérer le cas où on a utilisé 
une seule fois un de ces nouveaux axiomes. On suppose donc qu'on a une 
égalité  $ u_1(u_2(u_3+j_3(x))+j_2(x))+j_1(x)= 0$ où $x$ est un 
nouveau paramètre et en présence d'un nouvel axiome $P(x)=0$ où 
$P$ est un \pol unitaire à coefficients dans $\gK$. En fait~$j_1(x)$ est explicité sous forme $\sum_k j_{1,k}R_{1,k}(x)$ avec les 
$j_{1,k}$ dans $\cI_1$. Même chose pour~$j_2(x)$ et~$j_3(x)$.   
Après division des $R_{i,j}$ par $P$ relativement à la variable 
$x$, on obtient une égalité 
$ u_1(u_2(u_3+j'_3(x))+j'_2(x))+j'_1(x)= P(x)Q(x)$ avec  
$\deg(j'_i[x])<\deg(P)$ ($i=1,2,3$). Cela implique que $Q$ est 
identiquement nul. Il reste alors à considérer le coefficient 
constant de cette identité lorsqu'on  voit le premier membre comme un 
\pol en $x$.
\end{proof}

\begin{remark} 
\label{rem.algclos}
{\rm Si on ne sait pas construire une cloture algébrique de $\gK$, on 
peut quand même faire l'évaluation dynamique de la structure $S'$ 
où on a rajouté les axiomes de cloture algébrique relatifs à 
$\gK$. Le théorème précédent admet alors la variante suivante: 
ou bien cette structure collapse et cela nous permet de produire une 
identité $(b)$ à coefficients dans $\gK$, ou bien on peut 
construire un système triangulaire d'équations et inéquations 
polynomiales à coefficients dans~$\gK$ tel que dans la branche 
correspondante de $S'$ (qui est une branche qui ne collapse surement 
pas) on ait la construction explicite des \ideps définissant les 
variétés irréductibles $V_i$, avec la démonstration que le point $(a)$ 
est bien vérifé.   
} 
\end{remark}

\section{Suites régulières et pseudo régulières } 
\label{sec regseq}

\begin{definition} 
\label{def.pseudoreg}
Soit  $(x_1,\ldots,x_\ell)$ dans un anneau commutatif $\gA$. 
\begin{itemize}
\item  On dit que la suite 
$(x_1,\ldots,x_\ell)$ est une {\em suite pseudo singulière} 
lorsqu'il existe \hbox{$a_1,\ldots,a_\ell\in \gA$} et $m_1,\ldots,m_\ell\in \NN$ 
tels que
$$\begin{array}{c} 
x_1^{m_1}(x_2^{m_2}\cdots(x_\ell^{m_\ell} (1+a_\ell x_\ell) + 
\cdots+a_2x_2) + a_1x_1) =   \\[.4em] 
x_1^{m_1} \cdots x_\ell^{m_\ell} +
a_\ell x_1^{m_1} \cdots x_{\ell-1}^{m_{\ell-1}}x_\ell^{1+m_\ell}  + 
\cdots+a_2x_1^{m_1} x_2^{1+m_2} + a_1x_1^{1+m_1} = 0
\end{array}$$
\item   On dit que la suite 
$(x_1,\ldots,x_\ell)$ est une {\em suite pseudo ré\-gu\-liè\-re} 
lorsque  pour tous $a_1,\ldots,a_\ell\in \gA$ et tous 
$m_1,\ldots,m_\ell\in \NN$ on~a\footnote{\label{nbpIneq}Dans un cadre \cofz, il est parfois préférable de considérer une relation d'inégalité $x\neq 0$ 
qui ne soit pas simplement l'impossibilité de $x=0$. Par exemple un 
nombre réel est dit $\neq 0$ lorsqu'il est inversible, \cad 
clairement non nul. Chaque fois que nous mentionnons une relation 
d'inégalité $x\neq 0$, nous supposons donc toujours implicitement 
que cette relation a été définie au préalable dans l'anneau que 
nous considérons. Nous demandons que cette relation soit une 
inégalité standard, \cad qu'elle puisse être démontrée 
équivalente à $\lnot(x=0)$ en utilisant le principe du tiers exclu. 
Nous demandons en outre que l'on ait constructivement  
$\; (x\neq 0,\; y=0)\; \Rightarrow\;  x+y\neq 0$,  
$\; xy\neq 0\Rightarrow \; x\neq 0,$ et
$\lnot(0\neq 0)$. Enfin  $x\neq y$ est défini par  $x-y\neq 0$.
En l'absence de précisions concernant  $x\neq 0$, on peut toujours 
considérer qu'il s'agit de la relation $\lnot(x=0)$. Lorsque l'anneau 
est un ensemble discret, \cad lorsqu'il possède un test d'égalité 
à zéro, on choisit toujours l'inégalité $\lnot(x=0)$. 
Néanmoins ce serait une erreur de principe grave de considérer que 
l'algèbre commutative ne doit travailler qu'avec des ensembles 
discrets.} 
$$ x_1^{m_1}(x_2^{m_2}\cdots(x_\ell^{m_\ell} (1+a_\ell x_\ell) + 
\cdots+a_2x_2) + a_1x_1) \neq  0
$$
\end{itemize}
\end{definition}
Une suite pseudo régulière ne contient ni inversible ni nilpotent. 
Dans le cas d'un anneau local  on se limite à examiner les cas avec 
les $x_i$ dans l'idéal maximal, et $(1+a_\ell x_\ell)$ peut être 
remplacé par $1$.

La proposition suivante est une conséquence immédiate du lemme 
\ref{lem.coll} qui caractérise le collapsus d'une \sad de type 
$\cP_{\ell,dir}$ (ici $\ell=4$)
\begin{proposition} 
\label{prop.pseudoreg} Soit $\gA$ un anneau commutatif et $(x_1,x_2,x_3)$ 
trois éléments de $\gA$. Soit $\gB$ la \sad de type $\cP_{4,dir}$ 
obtenue à partir du diagramme de $\gA$ comme anneau commutatif (avec le 
prédicat $x=0$), en rajoutant les relations suivantes dans la 
présentation 
$$ \U1(x_1),\; \J2(x_1),\; \U2(x_2),\; \J3(x_2),\; \U3(x_3),\; 
\J4(x_3)$$
Alors $\gB$ ne collapse pas \ssi la suite $(x_1,x_2,x_3)$ est pseudo 
régulière
\end{proposition}
\begin{proof}{Démonstration.} Vue la présentation de $\gB$, et vu le lemme 
\ref{lem.coll} un collapsus de $\gB$ donne une égalité dans $\gA$ du 
type
$$   u_1(u_2(u_3(u_4+j_4)+j_3)+j_2)+j_1= 0   $$
avec $u_1=x_1^{m_1}$, $u_2=x_2^{m_2}$, $u_3=x_3^{m_3}$, $u_4=1$, 
$j_1=0$, $j_2=a_1x_1$, $j_3=a_2x_2$, $j_4=a_3x_3$, c.-à-d.
$$   x_1^{m_1}(x_2^{m_2}(x_3^{m_3}(1+a_3x_3)+a_2x_2)+a_1x_1)= 0   $$
\end{proof}

Le lien avec les suites régulières est donné par la proposition 
suivante. 

\begin{proposition} 
\label{prop.regseq} Considérons un anneau commutatif $\gA$ et trois 
éléments $x_1,x_2,x_3$ de $\gA$. Si la suite $(x_1,x_2,x_3)$ est 
régulière alors elle est pseudo régulière.
\end{proposition}
\begin{proof}{Démonstration.}
Supposons la suite régulière et montrons qu'une égalité 
$   x_1^{m_1}(x_2^{m_2}(x_3^{m_3}(1+a_3x_3)+a_2x_2)+a_1x_1)= 0   $ est 
impossible.
Si on a une égalité de ce type, comme $x_1$ est \ndz on en déduit 
$   x_2^{m_2}(x_3^{m_3}(1+a_3x_3)+a_2x_2)+a_1x_1= 0   $.\\ 
Comme $x_2$ est \ndz modulo $x_1\gA$ on en déduit  
$x_3^{m_3}(1+a_3x_3)+a_2x_2=dx_1$.\\
Donc $x_3^{m_3}(1+a_3x_3)\in (x_1\gA+x_2\gA)$. Puisque $x_3$ est \ndz 
modulo $x_1\gA+x_2\gA$ on obtient  $1+a_3x_3=ex_1+e'x_2$ ce qui est 
contradictoire avec $x_1\gA+x_2\gA+x_3\gA\neq \gA$.
\end{proof}
La version ``\nst formel'' de la proposition 
\ref{prop.pseudoreg}, que l'on déduit immédiatement des théorème 
\ref{th.nstformel} et
\ref{th.nstformel2},  est le théorème suivant.
\begin{theorem} 
\label{th.pseudoreg} {\em (suites pseudo régulières et chaînes 
croissantes d'\idepsz).}
\begin{description}
\item [(en \maths classiques)] Dans un anneau $\gA$ une suite 
$(x_1,x_2,x_3)$ est pseudo régulière \ssi il existe quatre \ideps 
$P_1\subset P_2\subset P_3\subset P_4$ avec $x_1\in P_2\setminus P_1$, 
$x_2\in P_3\setminus P_2$ et $x_3\in P_4\setminus P_3$.
\item [(en \maths constructives)] Ce théorème est valable 
constructivement lorsque $\gA$ est un anneau de Lasker-Noether. (cf. 
\cite{MRR})
\end{description}
\end{theorem}
\section{Vers une théorie constructive de la dimension de Krull} 
\label{sec Krull}

Ainsi, le fait qu'un anneau est de dimension $\geq \ell$ équivaut (en 
\maths classiques)~à l'existence d'une suite pseudo régulière de 
longueur $\ell$. 
Il serait intéressant de développer un traitement constructif de la 
théorie de la dimension de Krull basé sur les suites  
pseudo régulières. On donnerait les définitions suivantes.
\begin{definition} 
\label{def.dimKrull}
Un anneau $\gA$ est dit {\em de dimension $-1$} lorsque $1_\gA=0_\gA$. Dans 
le cas contraire il est dit de dimension $\geq 0$. Soit maintenant 
$\ell\geq 1$. L'anneau est dit {\em de dimension $\geq \ell$} s'il 
existe une suite pseudo régulière de longueur $\ell$. Il est dit 
{\em de dimension $\leq \ell-1$} si toute suite  $(x_1,\ldots,x_\ell)$ 
de longueur $\ell$ dans $\gA$ est pseudo singulière. 
L'anneau est dit {\em de dimension $\ell$} s'il est~à la fois de 
dimension $\ge\ell$ et $\le\ell$.
Il est dit {\em  de dimension $<\ell$ } lorsqu'il est impossible qu'il 
soit de dimension $\geq \ell$.
\end{definition}

Notez que $\RR$ est un anneau local de dimension $<1$, mais qu'on ne
 peut pas prouver constructivement qu'il est de dimension $\leq 0$.

Un anneau est local zéro-dimensionnel \ssi on a
$$\forall x\in \gA\; \; \; x \; \; {\rm est\; inversible\; ou\;
nilpotent}
$$

Il serait souhaitable d'obtenir, sous des hypothèses convenables une 
forme constructive générale du Principal Ideal Theorem de Krull.

Ce pourrait être quelque chose du genre suivant.

\begin{theorem} 
\label{th.PIT} Soit $\gA$ un anneau local noetherien cohérent 
à idéaux détachables dont l'idéal maximal $\cM$ est de type
fini. Il existe $x_1,\ldots,x_\ell\in\cM$ et un entier 
$n\in \NN$ tels que
$$\cM^n \subset x_1 \gA + \cdots + x_\ell \gA 
$$
\ssi $\gA$ est de dimension 
$\leq  \ell$.
\end{theorem}

\subsubsection*{Dimension de Krull d'un anneau de \pols sur un corps
discret} 
Dans ce paragraphe, nous donnons une démonstration particulièrement 
élémentaire de la dimension de Krull d'un anneau de \pols sur un 
corps discret. Cette démonstration est également valable en \clama puisqu'on 
a établi le théorème \ref{th.pseudoreg}. 

Nous avons tout d'abord.
\begin{proposition} 
\label{propKrDimetDegTr} 
Soit $\gK$ un corps discret, $\gA$ une $\gK$-algèbre commutative, et 
$x_1$, \ldots, $x_\ell$ dans $\gA$  algébriquement dépendants sur 
$\gK$. Alors la suite
$(x_1,\ldots,x_\ell)$ est pseudo singulière.
\end{proposition}
\begin{proof}{Démonstration.}
Soit $Q(x_1,\ldots,x_\ell)=0$ une relation de dépendance algébrique
sur $\gK$. Ordonnons les monômes de $Q$ dont le coefficient est non nul 
selon l'ordre lexicographique. On peut supposer \spdg que le 
coefficient du premier monôme non nul (pour cet ordre) est égal à 
$1$. Si $x_1^{m_1}x_2^{m_2}\cdots x_\ell^{m_\ell}$ est ce monôme, il 
est clair que $Q$ s'écrit sous forme
$$ Q=x_1^{m_1}\cdots x_\ell^{m_\ell}+ 
x_1^{m_1}\cdots x_\ell^{1+m_\ell}R_\ell+
x_1^{m_1}\cdots x_{\ell-1}^{1+m_{\ell-1}}R_{\ell-1}+\cdots+
x_1^{m_1}x_2^{1+m_2}R_2+ x_1^{1+m_1}R_1
$$
et l'égalité  $Q=0$  fournit donc le collapsus recherché.
\end{proof}

On en déduit.
\begin{theorem} 
\label{thKDP} Soit $\gK$ un corps discret. La dimension de Krull de 
l'anneau $\gK[X_1,\ldots,X_\ell]$ est égale à $\ell$.
\end{theorem}
\begin{proof}{Démonstration.}
Vue la proposition \ref{propKrDimetDegTr} il suffit de vérifier que 
la suite $(X_1,\ldots,X_\ell)$ est pseudo régulière. Or elle est 
régulière. 
\end{proof}


\newpage\rdb
\setcounter{section}{1}
\def\thesection{\Alph{section}}  
\addcontentsline{toc}{section}{Annexe~: Théories dynamiques et \sads}
\section*{Annexe~: Théories dynamiques et \sads}\label{secsadyn}
\markboth{Annexe}{Annexe}
Cette annexe reprend, pour la commodité de la lecture, des 
définitions données dans \cite{clr} et \cite{lom98} (voir aussi 
\cite{lom95,lom97}).  

Une \sad est une {\em structure algébrique incomplètement 
spécifiée}. Pour cela on donne les axiomes que doit vérifier la 
structure, des générateurs et des relations. Le fait que la 
structure est {\em  algébrique} tient à ce que les axiomes ont une 
structure particulièrement simple. Ils peuvent être mis en œuvre 
à l'intérieur d'un calcul arborescent dans lequel la logique sous 
sa forme usuelle sophistiquée ne joue plus aucun role. Aucune formule 
compliquée avec connecteurs ou quantificateurs n'est mise en œuvre. 
Seuls apparaissent les prédicats donnés au départ dans la 
structure. En particulier, on peut éviter tout recours à la 
négation. Et le calcul mis en œuvre est, le long de chaque branche 
du calcul arborescent, de l'algèbre pure, en fait surtout de la 
réécriture. L'arborescence elle-même est gouvernée par des 
axiomes simples. Elle est pour l'essentiel une discussion cas par cas.

Les méthodes dynamiques pour l'interprétation constructive de 
résultats et de méthodes d'algèbre abstraite sont inspirées de 
la clôture algébrique dynamique implantée en AXIOM suite aux 
travaux de Dominique Duval (cf. \cite{ddd,dd,dr1,dr2})

Nous passons maintenant à des définitions plus formelles. 

\paragraph{Une  théorie dynamique} $\cT=(\cL\vert\cD)$ est donnée 
par \npb
\label{para thdy}

\begin{itemize}
\item Un langage $\cL$ comportant des constructeurs de termes (les 
termes sont construits à partir de variables, de paramètres, de 
constantes et de symboles de fonctions précisés), et des symboles 
de prédicat, dont au moins l'égalité.  (Notez que les seules 
formules à notre disposition sont les formules atomiques puisque le 
langage ne comporte ni connecteur ni quantificateur.)
\item  Un ensemble $\cD$ de règles de déduction, appelées 
axiomes, de nature élémentaire (on va préciser cela), formulés 
sans utilisation des paramètres.
\end{itemize}

\ss Les axiomes sont des {\em règles dynamiques}, c.-à-d. des 
règles de déduction du type général suivant~: 

\axio{H(\gx) }{  \left(\E \gy^1 A_1(\gx,\gy^1) \ou \ldots \ou \E \gy^k
A_k(\gx,\gy^k)\right)}{{\rm Nom-Ax}}

\noi    
Les $ A_i(\gx,\gy^i)$ et $ H(\gx)$ sont des listes de
prédicats portant sur des termes construits avec le langage de 
départ, {\em  sans utilisation des paramètres{\footnote{Nous 
utiliserons souvent par abus de notation les mêmes lettres pour 
désigner les variables (présentes uniquement dans les axiomes) et 
les paramètres (présents dans les faits, mais non dans les 
axiomes). La distinction entre paramètres et variables est en fait 
bien utile pour éviter les précautions usuelles difficiles à 
formuler qui gouvernent la légitimité du remplacement d'une 
variable muette par une variable non muette. Notre point de vue est 
d'appeler paramètre toute variable non muette et d'adopter des noms 
radicalement distincts pour les variables et pour les paramètres.}}}. 
Un $ \forall \gx $ est implicite devant l'axiome. Le $\E$ a la 
signification d'un ``il existe'', la virgule dans une liste a la 
signification d'un ``et''. Les variables présentes dans la règle 
sont d'une part les $x_h$ dans $\gx$, d'autre part les $y^i_j$ dans les 
$\gy^i$. Aucune variable $y^i_j$ n'est une variable $x_h$. Mais on peut 
avoir des variables communes dans deux listes $\gy^i$, car les 
$\E \gy^i A_i(\gx,\gy^i)$ sont indépendants les uns des autres.

Un axiome est dit {\em disjonctif}  s'il y a effectivement des  
``$\ou$'', il est dit {\em existentiel}  s'il y a effectivement des 
variables dans une liste $\gy^i$  (c.-à-d. si les listes $\gy^i$ ne 
sont pas toutes vides). 
Un axiome qui n'est ni disjonctif ni existentiel est dit {\em 
universel} ou {\em purement algébrique}.
Une \tdy est dite {\em purement algébrique} si tous ses axiomes sont 
purement algébriques.

Dans la liste des axiomes, il y a les {\em  axiomes logiques} usuels 
relatifs à l'égalité, écrits sous forme de \rdysz.

\paragraph{Une structure algébrique dynamique}\npb  $S=(\cT,\cP)$ est 
donnée par~:
\label{para sady}

\sni --- une \tdy $\cT=(\cL\vert\cD)$ qui définit {\em le type 
abstrait de la structure algébrique}, d'une part,
 
\sni --- une {\em présentation $\cP=(G\vert R)$ de la structure 
dynamique} $S$ d'autre part, donnée comme suit
\begin{itemize}
\item Un ensemble $G$ de {\em générateurs}, qu'on peut voir comme 
de nouvelles constantes. Nous dirons parfois qu'il s'agit des 
paramètres de départ, particulièrement s'ils ne figurent dans 
aucune relation. Naturellement $\cL\cap G=\emptyset$.
\item Et un ensemble de relations $R$ qui sont des {\em formules 
(atomiques) $S$-closes}, c.-à-d. des prédicats de $\cL$ portant sur 
des {\em  termes $S$-clos} (c.-à-d. des termes construits à partir 
des constantes de $\cL$ et des générateurs de $S$ au moyen des 
symboles de fonction de $\cL${\footnote{Un terme $S$-clos ne contient 
donc ni variables ni paramètres (hormis les générateurs si on les 
considère comme des paramètres de départ.) Nous dirons parfois 
{\em terme clos de $S$} à la place de ``terme $S$-clos''.}}). On peut 
voir les relations dans $R$ comme~:  les faits donnés au départ 
comme vrais dans la structure. On les appellera les {\em relations de 
départ}.
\end{itemize}

Un {\em fait} concernant une \sad 
$S=(\cT, \cP)=((\cL\vert \cD),(G\vert R))$ est une affirmation~: 

\sni\centerline{
 Tel(s) terme(s) $S$-clos satisfait(font) tel prédicat (donné dans 
la structure).
}
  
\sni On pourrait aussi dire que c'est un ``fait brut''{\footnote{De 
manière générale, l'ensemble des assertions exprimables sous 
forme de faits dépend de manière critique du langage utilisé. Un 
langage plus riche permet d'exprimer plus de propriétés sous forme 
de faits. Par exemple si pour la structure d'anneau, on n'introduisait 
ni la constante $-1$, ni l'opération $ x \mapsto -x$, on serait 
contraint~à des périphrases bien encombrantes.}}.  
Nous donnons plus loin une définition un peu plus formelle de la 
même chose. Nous avons besoin de préciser auparavant ce que sont 
les démonstrations dynamiques.
\paragraph{Utilisation légitime d'un axiome comme \rdyz}~
\label{para axiomes}

\noi De même qu'une théorie purement équationnelle débouche sur 
un {\em calcul algébrique pur}, 
n'u\-ti\-li\-sant rien d'autre que de la réécriture de manière 
systématique (le calcul ne fait pas intervenir de logique si ce 
n'est, implicitement, au travers des axiomes relatifs à 
l'égalité), de même une \tdy ne doit pas être vue au premier 
chef comme une variante des théories logiques du premier ordre, mais 
comme une {\em pure machinerie calculatoire arborescente}.  
 
\ss Les axiomes n'ont pas pour signification de donner des énoncés 
élémentaires vrais, mais {\em d'être utilisés comme règles de 
déduction}, ou plus précisément comme {\em  règles de 
constructions d'\evdsz}. C'est pour cette raison, et non par 
provocation, que nous n'avons utilisé dans les \rdys aucun des 
symboles usuels servant à écrire des formules du premier ordre 
construites à partir des prédicats donnés au départ.

\ss
{\em Une \evd d'une \sadz}  
$S=(\cT,(G\vert R))$ est un arbre. A la racine de l'arbre, il y a 
implicitement présents tous les faits concernant $S$ contenus dans 
les relations de départ. En chaque point de l'arbre des {\em 
paramètres existentiels} peuvent avoir été introduits lors de 
l'utilisation d'axiomes existentiels. Un terme construit sur 
$\cL\cup G$ est appelé un {\em terme clos} en un point de l'arbre 
s'il ne contient pas de variable et si les seuls paramètres sont des 
générateurs de $S$ ou des paramètres existentiels 
précédemment introduits dans la branche où on est.
 
Lors de la construction d'un arbre d'\evdz, voici ce qu'il faut faire 
pour utiliser {\em légitimement} un axiome tel que~: 

\axio{H(\gx) }{  \left(\E \gy^1 A_1(\gx,\gy^1) \ou \ldots \ou \E \gy^k 
A_k(\gx,\gy^k)\right)}{{\rm Nom-Ax}}

\noi On remplace les $x_j$ de la liste $\gx$ par des termes $t_j$ qui 
sont clos (au point où on est) et dont on a établi qu'ils valident 
l'hypothèse $H(\gx)$ (au point où on est). Si l'axiome est 
disjonctif, on produit un nœud à $k$ branches. Dans la $i$-ème 
branche, les formules constituant la liste 
$ A_i({\bf t },\gy^i)$ sont valides. Si l'axiome est existentiel, on 
introduit de nouveaux paramètres correspondant aux objets affirmés 
exister dans l'axiome. Les $y^i_j$ doivent être pris parmi les 
paramètres non encore utilisés dans la branche où l'on est.

\ss Dans une \evd d'une \sad $S$ les seuls pa\-ra\-mètres qui 
apparaissent sont donc les paramètres de départ et les 
paramètres existentiels.

\ss En un point d'un arbre d'\evd d'une structure $S$ on appelle 
{\em fait} une formule atomique $Q(t_1,\ldots,t_n)$ où $Q$ est un 
prédicat $n$-aire du langage $\cL$ et les $t_i$ sont des termes clos 
en ce point de l'arbre. 
Un arbre d'\evd est donc un système de démonstration extrêmement simple 
visant à valider des faits.

\begin{definition} {\em  (faits vrais dans une \sadz)}
\label{def.vrai} Soit $S=(\cT,\cP)=((\cL\vert\cD),(G\vert R))$ une 
\sadz.
Un {\em  fait concernant $S$} (ou encore, un fait de $S$) est une 
formule $Q(t_1,\ldots,t_n)$ où $Q$ est un prédicat figurant dans 
$\cL$ et les $t_i$ sont des termes $S$-clos de $\cL\cup G$.
Un tel fait est dit {\em vrai dans la \sad} $S$ considérée si on a 
construit un arbre d'\evd de la structure, et que ce fait est prouvé 
vrai à toutes les feuilles de l'arbre.
\end{definition}

Une chose importante est qu'une \evd ne s'intéresse qu'aux choses 
vraies, et pas aux choses fausses. 
Un fait qui n'est pas prouvable ne peut pas pour autant être 
déclaré faux. 
Nous examinerons ceci plus en détail un peu plus loin.

\ss Toute structure algébrique ordinaire  $S$  qui satisfait les 
axiomes d'une \tdy $\cT$ (lus comme des axiomes ordinaires) fournit un 
cas particulier de \sad  $S_1=(\cT,(G\vert {\rm  Diag}_{\cT}(S)))$, 
en prenant comme générateurs de départ les éléments de  $S$  
(ou un système de générateurs de  $S$)  et comme relations de 
départ le $\cT$-diagramme de $S$, ${\rm  Diag}_{\cT}(S)$, \cad les 
faits  vrais dans  $S$  (ou un ensemble de faits vrais dans  $S$  qui 
impliquent, au sens de l'\evdz, tous les faits vrais dans  $S$). 
Nous noterons $\sd{S}{\cT}$ cette \sadz.

Les faits  vrais (au sens usuel) dans la structure ordinaire  $S$  sont 
alors exactement les faits  vrais (au sens de l'\evdz) dans la 
structure dynamique  $\sd{S}{\cT}$. A vrai dire dans ce cas, 
le processus de l'\evd n'apporte rien concernant les faits. 
Si au contraire la  structure algébrique ordinaire $S$ ne satisfait 
pas certains axiomes de $\cT$, en particulier si elle ne possède pas 
tous les prédicats du langage de $\cT$, le processus d'\evd peut 
enclancher une exploration des structures correspondantes 
``idéalement possibles'' formées à partir de $S$ quand on impose 
les contraintes supplémentaires de $\cT$.

\ss Une {\em \rdyz}  a la même forme générale qu'un axiome~:
$$ H(\gx) \impl \left(\E\gy^1 A_1(\gx,\gy^1) \ou \ldots \ou \E\gy^k
A_k(\gx,\gy^k)\right) $$ 
Elle est dite {\em  valide (pour la \tdy considérée)} si elle peut 
être prouvée à partir des
axiomes de la théorie. Cette démonstration doit être obtenue en donnant 
une \evd de la structure $S$ de type $\cT$ et de présentation 
$\cP=(\gu\vert H(\gu))$ (les variables $x_i$ ont été remplacées 
par des générateurs $u_i$). A
l'extrémité de chaque branche doit être prouvée l'une des
conclusions 
$A_i(\gu,\gt^i)$ où les $t^i_j$ sont des termes construits au cours 
de l'\evdz.
Lorsqu'on a prouvé qu'une règle est valide, son utilisation de la 
même manière que les axiomes est légitime et elle ne permet pas 
de démontrer d'autres faits que ceux prouvés à partir des 
axiomes.

\paragraph{La structure réduite à un point et le collapsus}~\npb
\label{para coll}

\noi  La {\em structure ponctuelle} est par définition la structure 
réduite à un point en lequel tous les prédicats sont vrais. Cette 
structure satisfait toujours les axiomes, ce qui évite radicalement 
l'usage de la négation.

\smallskip Dans une \evdz, une branche {\em meurt}, ou {\em collapse}, 
si elle prouve un fait à partir duquel on sait déduire que tous les 
faits sont vrais (ce qui correspond à la structure ponctuelle 
précédemment décrite). 
Dans le cas de structures qui sont des surstructures de la structure 
d'anneau, un tel fait est  $1=0$, (on rajoutera s'il le faut un axiome 
affirmant, pour chaque nouveau prédicat, qu'il est vrai sous 
l'hypothèse  $1 = 0$).
Dans le cas général, on introduit  $\Faux$  en tant que constante 
logique, avec les {\em axiomes logiques} (un pour chaque prédicat 
$n$-aire)~:

\axio{\Faux}{A(x_1,x_2,\ldots,x_n)}{\Faux,A}

\noi (en particulier cela réduit la structure à un point). 
La connotation inexistentielle du  Faux  est donc intuitivement 
remplacée par la connotation de ``réduction au cas trivial sans 
intérêt''.  

\begin{definition} 
\label{def.collapsus} 
\begin{itemize}
\item Une structure algébrique dynamique {\em collapse}  si on a 
construit une évaluation dynamique dans laquelle toutes les branches 
sont mortes. 
\item Soient $\cT_1=(\cL_1\vert\cD_1)$ et $\cT_2=(\cL_2\vert\cD_2)$ 
deux \tdys basées sur des extensions $\cL_1$ et $\cL_2$ d'un même 
langage $\cL=\cL_1\cap\cL_2$. 
\begin{itemize}
\item  On dit que $\cT_1$ et $\cT_2$ {\em  collapsent simultanément} 
si pour toute présentation $\cP=(G\vert R)$ sur le langage $\cL$, la 
structure $S_1=(\cT_1,\cP)$ collapse \ssi la structure 
$S_2=(\cT_2,\cP)$ collapse. 
\item  On dit que les $\cT_1$ et $\cT_2$ {\em  prouvent les mêmes 
faits} si pour toute présentation $\cP=(G\vert R)$ sur le langage 
$\cL$, la structure $S_1=(\cT_1,\cP)$ et la structure $S_2=(\cT_2,\cP)$ 
prouvent les mêmes faits (écrits dans $\cL\cup G$).
\end{itemize}
\end{itemize}
\end{definition}

Lorsque  $\cT_2$ est une extension de $\cT_1$  les axiomes et 
prédicats supplémentaires introduits dans $\cT_2$ pour nous 
faciliter les calculs (il est plus facile de calculer dans un corps 
réel clos que dans un corps ordonné par exemple) sont en fait 
inoffensifs, s'il y a collapsus simultané.
Les démonstrations de \nsts dans \cite{clr} sont obtenues notamment par la 
mise en évidence que certaines \tdys collapsent simultanément. 
Cette mise en évidence est elle même très proche de certaines 
démonstrations classiques abstraites des mêmes   
\nstsz. La différence avec les démonstrations classiques en question est que 
dans \cite{clr} tous les \nsts sont effectifs.

\newpage

\markboth{Note historique}{Note historique}
\rdb
{\bf \large Note historique}
\addcontentsline{toc}{section}{Note historique} \label{notehistorique}

\smallskip L'article présent donne avec les \dfns \ref{def.pseudoreg} et 
\ref{def.dimKrull} une \dfn constructivement acceptable des phrases \gui{la \ddk de l'anneau  $\gA$ est $\leq \ell$}\  et  \gui{la \ddk de l'anneau  $\gA$ est $\geq \ell$}. Bien que l'introduction insiste plus sur les \thos de type \nstz, la section \ref{sec Krull} annonce clairement le projet d'utiliser cette \dfn pour donner un contenu algorithmique aux (à la plupart des) \thos de \clama qui ont pour hypothèse une majoration de la \ddk et pour conclusion un résultat concret. Ce programme a été en partie réalisé, essentiellement grâce à Thierry Coquand, comme on peut s’en rendre compte en lisant le survey  \url{www.cse.chalmers.se/~coquand/logicalgebra.pdf} et les derniers chapitres du livre  {\it Commutative algebra: Constructive methods. Finite projective modules}. Henri Lombardi, Claude Quitté. Springer (2015): \url{https://arxiv.org/abs/1605.04832} (version enrichie et corrigée).

\smallskip Les définitions constructives de la dimension de Krull d'un anneau commutatif ont été mises au point dans plusieurs articles. 

\smallskip Le premier est la note d'André Joyal: Les théorèmes de Chevalley-Tarski et remarques sur l'algèbre constructive. \emph{Cah. Topologie Géom. Différ. Catégoriques}, {\bf 16}, 256--258, (1976), développée dans la thèse de Luis Espa{\~ n}ol: Dimensión en álgebra constructiva. 1978. 
\url{https://dialnet.unirioja.es/descarga/tesis/1402.pdf} et dans plusieurs articles, notamment: Luis Espa{\~n}ol:  Le spectre d'un anneau dans l'algèbre constructive et applications à la dimension. \emph{Cah. Topologie Géom. Différ. Catégoriques}, {\bf 24}, 133--144, (1983).

La définition de Joyal évoquée dans Boileau and Joyal (1981) a été analysée par Cederquist and Coquand (2000) en utilisant la notion de relation implicative.\\
André Boileau and André Joyal. La logique des topos.
\emph{Journal of Symbolic Logic}, {\bf 46}, 6--16, (1981).\\
Jan Cederquist and Thierry Coquand. Entailment relations and distributive lattices. \emph{Logic Colloquium '98 ({P}rague)}, volume~13 des \emph{Lect.
  Notes Log.}, pages 127--139. Assoc. Symbol. Logic, Urbana, IL, (2000).
\url{https://www.cse.chalmers.se/~coquand/lattice.ps}

\smallskip Le deuxième est l'article présent, écrit sans avoir connaissance  des travaux de Joyal,  Español et Coquand. Dans cet article une caractérisation explicite de la dimension en termes d’identités algébriques est démontrée. Une caractérisation de ce type n'existait pas dans les travaux précédents.
Cette caractérisation purement algébrique conduit à une démonstration particulièrement simple de la dimension d'un anneau de polynômes sur un corps, donnée dans le \tho \ref{thKDP}.

\smallskip Ensuite, deux articles donnent 
l'explication de l'équivalence des deux notions précédentes en mathématiques constructives et développent pour cela une troisième caractérisation constructive équivalente.\\ 
T. Coquand, H. Lombardi. Hidden constructions in abstract algebra: Krull dimension of distributive lattices and commutative rings. In \emph{Commutative ring theory and applications} (Fez, 2001), volume 231 de Lecture Notes in Pure and Appl. Math., 477--499. Dekker, New York, 2003. \url{http://arxiv.org/abs/1712.04725}.\\
T. Coquand, H. Lombardi.  Constructions cachées en algèbre abstraite. Dimension de Krull, Going up, Going down. Rapport technique, Département de Mathématiques de l'Université de Franche-Comté, 2018.  Mise à jour en 2018 d'un preprint de 2001. \url{lhttp://arxiv.org/abs/1712.04728}.

\smallskip Enfin un troisième article donne une caractérisation en termes de \gui{bords}:\\
Thierry Coquand, Henri Lombardi et Marie-Françoise Roy: An elementary characterization of Krull dimension. Dans \emph{From sets and types to topology and analysis}, volume 48 de Oxford Logic Guides, 239--244. Oxford Univ. Press, Oxford, 2005. \url{http://hlombardi.free.fr/publis/lebord.pdf}.

\smallskip 
Signalons dans la suite immédiate de l’article précédent les articles:
\\
T. Coquand, H. Lombardi. A short proof for the Krull dimension of a polynomial ring. 
\emph{American Math. Monthly}   {\bf 112}  no. 9, 826--829,  (2005).
\url{http://hlombardi.free.fr/publis/KrullMathMonth.pdf}.
\\
T. Coquand. {Sur un th\'eor\`eme de {K}ronecker concernant les vari\'et\'es
              alg\'ebriques},
{\em Comptes Rendus Math\'ematique. Acad\'emie des Sciences. Paris},
{\bf 338}, no4, 291--294, (2004).
\url{https://www.cse.chalmers.se/~coquand/heit.ps}.

\smallskip La définition en termes de bords est récursive et c'est elle qui a permis de donner des versions constructives de nombreux résultats classiques par la suite. Notamment le Splitting Off de Serre, le théorème de Forster majorant le nombre de générateurs d'un module de type fini et le théorème de simplification de Bass, tout ceci dans le cas où l'hypothèse est une borne sur la dimension de Krull. 

\smallskip La version non noethérienne de ces théorèmes est due à Raymond Heitmann dans:\\    
Generating non-Noetherian modules efficiently. \emph{Mich. Math. J.}, {\bf 31}, 67--180, 1984.

La nouvelle caractérisation récursive de la dimension de Krull a permis de traduire les démonstrations de Heitmann en des algorithmes, décrits dans l'article: \\
Thierry Coquand, Henri Lombardi et Claude Quitté: Generating non-Noetherian modules constructively. \emph{Manuscripta Math.}, {\bf 115} (4) 513--520, 2004. \url{https://www.cse.chalmers.se/~coquand/fs.ps}

\smallskip Une autre issue insoupçonnée de la dimension définie dans l'article présent a été mise à jour par G. Kemper et I. Yengui dans l'article suivant:\\
Valuative dimension and monomial orders. \textsl{Journal of Algebra}, {\bf 557}, 278--288, 2020. \url{https://arxiv.org/abs/1906.12067} (à la suite de \url{https://arxiv.org/pdf/1303.3937.pdf}).
\\
Cet article donne une définition constructive de la dimension valuative d'un anneau commutatif basée sur une très légère variante de la définition de la dimension de Krull donnée dans l'article présent.
 Ihsen Yengui a par ailleurs utilisé la  définition constructive de la  
dimension de Krull dans plusieurs articles, dont certains établissent des résultats inconnus auparavant en mathématiques classiques.

\tableofcontents

\end{document}